\documentclass[11pt, a4paper]{amsart}
\usepackage{fullpage}
\usepackage{amssymb}
\usepackage{commath}
\usepackage{mathrsfs}
\usepackage{empheq}
\usepackage{amsmath, amsthm}
\usepackage{multicol}
\usepackage{parskip}
\usepackage{graphicx}
\usepackage{realboxes}
\usepackage[colorlinks=true, allcolors=blue]{hyperref}

\usepackage{comment}

\usepackage{tikz,tikz-3dplot,tikz-cd,tkz-tab,tkz-euclide,pgf,pgfplots}
\usetikzlibrary{bending,quotes,angles,3d}
\usepackage{bbm}
\usepackage{url}

\newcommand{\ol}{\overline}
\newcommand{\ul}{\underline}
\renewcommand{\phi}{\varphi}
\newcommand{\llb}{\left\lbrace}
\newcommand{\rrb}{\right\rbrace}

\newcommand{\llv}{\left\lvert}
\newcommand{\rrv}{\right\rvert}
\newcommand{\1}{\mathbbm{1}}
\newcommand\fakeqed{\pushQED{\qed}\qedhere}

\DeclareMathOperator{\Tor}{Tor}
\DeclareMathOperator{\Ext}{Ext}
\DeclareMathOperator{\Hom}{Hom}

\newtheorem{thm}{Theorem}[section]
\newtheorem{abcthm}{Theorem}

\theoremstyle{plain}
\newtheorem{prop}[thm]{Proposition}
\newtheorem{lem}[thm]{Lemma}
\newtheorem{cor}[thm]{Corollary}

\newtheorem*{theorem*}{Theorem}

\theoremstyle{definition}
\newtheorem{defn}[thm]{Definition}

\theoremstyle{remark}
\newtheorem{rem}[thm]{Remark}
\newtheorem{eg}[thm]{Example}

\numberwithin{thm}{subsection}

\makeatletter
\@namedef{subjclassname@1991}{\textup{2020} Mathematics Subject Classification}
\makeatother

\title{Cohomology of Tanabe algebras}
\author{Andrew Fisher}
\address{(Andrew Fisher) School of Mathematical and Physical Sciences, University of Sheffield, Hounsfield Road, S3 7RH, UK}
\email{afisher1@sheffield.ac.uk}
\author{Daniel Graves}
\address{(Daniel Graves) Lifelong Learning Centre, University of Leeds, Woodhouse, Leeds, LS2 9JT, UK}
\email{dan.graves92@gmail.com}
\date{}

\begin{document}



\begin{abstract}
In this paper we study the (co)homology of Tanabe algebras, which are a family of subalgebras of the partition algebras exhibiting a Schur--Weyl duality with certain complex reflection groups. The homology of the partition algebras has been shown to be related to the homology of the symmetric groups by Boyd--Hepworth--Patzt and the results they obtain depend on a parameter. In all known results, the homology of a diagram algebra is dependent on one of two things: the invertibility of a parameter in the ground ring or the parity of the positive integer indexing the number of pairs of vertices. We show that the (co)homology of Tanabe algebras is isomorphic to the (co)homology of the symmetric groups and that this is independent of both the parameter and the parity of the index. To the best of our knowledge, this is the first example of a result of this sort. Along the way we will also study the (co)homology of uniform block permutation algebras and totally propagating partition algebras as well collecting cohomological analogues of known results for the homology of partition algebras and Jones annular algebras.
\end{abstract}

\maketitle

\section{Introduction}

The homology of diagram algebras is an emerging area of study. Diagram algebras are examples of augmented algebras and their homology and cohomology can be defined in terms of certain $\Tor$ and $\Ext$ groups following \cite[Definition 2.4.4]{Benson1}. Examples of algebras that have been studied include the Temperley--Lieb algebras \cite{BH1,Sroka}, the Brauer algebras \cite{BHP}, the partition algebras \cite{BHP2, Boyde2}, the Jones annular algebras \cite{Boyde2}, the rook algebras and the rook-Brauer algebras \cite{Boyde}.

The partition algebras, $P_n(\delta)$ (where $n$ is a positive integer and $\delta$ is a parameter in a unital, commutative ground ring), were introduced independently by Martin \cite{Martin-partition} and Jones \cite{Jones-partition} to study the Potts model. Jones showed that the partition algebras exhibit a \emph{Schur--Weyl duality} with the symmetric groups, so called because it takes a similar form to the classical statement of Schur--Weyl duality between the symmetric groups and the general linear groups. We note below that other diagram algebras exhibit similar Schur--Weyl dualities. 

Loosely speaking, a partition $n$-diagram is an undirected graph on two columns of vertices, where the connected components determine and are determined by a partition of the vertices. The algebra $P_n(\delta)$ is spanned linearly by such diagrams with a product given by composing partitions (recalled in Section \ref{alg-sec} below).

 The homology of partition algebras has been studied by Boyd, Hepworth and Patzt \cite{BHP2}. Their results split into two cases, which depend on the parameter $\delta$. They show that if the parameter $\delta$ is invertible in the ground ring, then the homology of the partition algebras is globally isomorphic to the homology of the symmetric groups. However, if $\delta$ is not invertible then the homology of the partition algebras is only known to be isomorphic to the homology of the symmetric groups in a range. This is a recurring phenomenon in this very young field. The homology of the Temperley--Lieb algebras \cite{BH1} and the homology of the Brauer algebras \cite{BHP} both exhibit similar parameter-dependent behaviour. The homology of some diagram algebras also exhibit different behaviours depending on the parity of the index $n$, which counts the number of pairs of vertices in a basis element. Such results exist for the homology of Temperley--Lieb algebras (see \cite[Theorem A]{Sroka}) and the Brauer algebras (see \cite[Theorem 1.3]{Boyde}).

We will study the homology and cohomology of three subalgebras of the partition algebra: the \emph{Tanabe algebras}, the \emph{uniform block partition algebras} and the \emph{totally propagating partition algebras}.

Tanabe \cite{Tanabe} introduced a family of subalgebras, $\mathcal{T}_n(\delta,r)$, of the partition algebra. For each positive integer $r$, $\mathcal{T}_n(\delta,r)$ is spanned by partition $n$-diagrams where, for each connected component, the difference between the number of vertices in each column is zero modulo $r$. Tanabe demonstrated that these subalgebras exhibit a Schur--Weyl duality with certain complex reflection groups. Our treatment of the Tanabe algebras will follow that of Orellana \cite{Orellana}. 

We note that for any $r>n$, this condition is the same. It dictates that every connected component must have the same number of vertices in each column. In this way we obtain the uniform block permutation algebra, $U_n$, originally studied by Kosuda under the name \emph{party algebra} \cite{Kosuda1,Kosuda3}. The name uniform block permutations was coined by FitzGerald \cite{FitzGerald}.

The totally propagating partition algebras, $TPP_n$, are spanned linearly by partition $n$-diagrams such that every connected component contains vertices from each column in the graph. Kudryavtseva and Mazorchuk \cite{KM-TPP} have shown that the totally propagating partition algebras exhibit a Schur--Weyl duality with the rook algebras (see \cite{HalvR} for further details of the rook algebra).

Our main result is as follows.

\begin{abcthm}
\label{Tanabe-thm}
Let $\delta \in k$. Let $r\geqslant 2$. There exist isomorphisms of graded $k$-modules
\begin{enumerate}
\item\label{Tanabe-thm-1} $\Tor_{\star}^{\mathcal{T}_n(\delta , r)}\left(\1 , \1\right) \cong H_{\star}(\Sigma_n , \1)$ and $\Ext_{\mathcal{T}_n(\delta , r)}^{\star}\left(\1 , \1\right) \cong H^{\star}(\Sigma_n , \1)$;
\item $\Tor_{\star}^{U_n}\left(\1 , \1\right) \cong H_{\star}(\Sigma_n , \1)$ and $\Ext_{U_n}^{\star}\left(\1 , \1\right) \cong H^{\star}(\Sigma_n , \1)$ and
\item $\Tor_{\star}^{TPP_n}\left(\1 , \1\right) \cong H_{\star}(\Sigma_n , \1)$ and $\Ext_{TPP_n}^{\star}\left(\1 , \1\right) \cong H^{\star}(\Sigma_n , \1)$.
\end{enumerate}
\end{abcthm}

In particular, we obtain isomorphisms which are \emph{independent} of both the parameter $\delta$ and the index $n$, showing that the (co)homology of these subalgebras behaves differently to the (co)homology of the partition algebra. 

Theorem \ref{Tanabe-thm} follows from a technical lemma about subalgebras of the partition algebra (Lemma \ref{idempotent-cover-A-lem}) and the following theorem (of which the homological part is \cite[Theorem 1.7]{Boyde2}, whilst the cohomological part follows from Proposition \ref{partition-technical-prop} below).

\begin{abcthm}
\label{theorem-B}
Let $A$ be an augmented $k$-algebra with trivial module $\mathbbm{1}$. Let $I$ be a two-sided ideal of $A$ which is free as a $k$-module and which acts as multiplication by $0\in k$ on $\mathbbm{1}$. Suppose that there exists a $k$-free idempotent left cover of $I$ of height $h$ and width $w$. There are natural isomorphisms of $k$-modules
\[\Tor_q^A(\mathbbm{1},\mathbbm{1}) \cong \Tor_q^{A/I}(\mathbbm{1},\mathbbm{1}) \quad \text{and} \quad \Ext_A^q(\mathbbm{1},\mathbbm{1}) \cong \Ext_{A/I}^q(\mathbbm{1},\mathbbm{1})\]
for $q\leqslant h$. Furthermore, the natural maps  \[\Tor_{h+1}^A(\1,\1)\rightarrow \Tor_{h+1}^{A/I}(\1,\1) \quad \text{and} \quad \Ext_{A/I}^{h+1}(\1,\1) \rightarrow \Ext_{A}^{h+1}(\1,\1)\] are a surjection and an injection respectively.

Finally, if $h=w$, then we have natural isomorphisms of graded $k$-modules
\[\Tor_{\star}^A(\mathbbm{1},\mathbbm{1}) \cong \Tor_{\star}^{A/I}(\mathbbm{1},\mathbbm{1}) \quad \text{and} \quad \Ext_A^{\star}(\mathbbm{1},\mathbbm{1}) \cong \Ext_{A/I}^{\star}(\mathbbm{1},\mathbbm{1}).\]
\end{abcthm}

We will use techniques which were developed in \cite{Boyde2} to study the homology of partition algebras and Jones annular algebras. These techniques were themselves inspired by work of Sroka \cite{Sroka}. Let $A$ be a Tanabe algebra, a uniform block permutation algebra or a totally propagating partition algebra. Let $I$ be the two-sided ideal spanned $k$-linearly by non-permutation diagrams and observe that in each case there is an isomorphism of $k$-algebras $A/I\cong k[\Sigma_n]$. We will define a family of left ideals $L_{i,j}$ which cover the ideal $I$ and such that any intersection of these ideals is either zero or is principal and generated by an idempotent. Such a family of ideals is called a \emph{$k$-free idempotent left cover with height equal to its width}, in the terminology of \cite{Boyde2}. With this in place, Theorem \ref{Tanabe-thm} follows directly from Theorem \ref{theorem-B}. The key step is that a $k$-free idempotent left cover allows us to define a chain complex called the \emph{Mayer--Vietoris complex}. In the case of the partition algebras, Boyde shows that this is a partial resolution of $k[\Sigma_n]$ by $P_n(\delta)$-modules to obtain isomorphisms between the homology of the partition algebras and the homology of the symmetric groups in a range. In our case, we can use the Mayer--Vietoris complex to construct a genuine projective resolution of $k[\Sigma_n]$ by $A$-modules.

These techniques will also allow us to deduce cohomological versions of Boyde's results for the Jones annular algebras and the partition algebras. In particular, we deduce a cohomological stability result for the partition algebras after the fashion of \cite[Corollary C]{BHP2}.

The paper is structured as follows. In Section \ref{alg-sec}, we recall the definitions of partition algebras, Tanabe algebras, uniform block permutation algebras and totally propagating partition algebras. In Section \ref{Theorem-B-section} we will recall Boyde's notion of $k$-free idempotent left cover and the Mayer--Vietoris complex. We use these to prove Theorem \ref{theorem-B}. In Section \ref{Cohom-partition-sec} we use Theorem \ref{theorem-B} to prove the cohomological analogues of Boyde's results for the partition algebras and Jones annular algebras. Finally, in Section \ref{Tanabe-sec} we prove Theorem \ref{Tanabe-thm}.

\subsection*{Acknowledgements}
We would like to thank James Brotherston and Natasha Cowley for helpful and interesting conversations whilst writing this paper. We would like to thank James Cranch and Sarah Whitehouse for the their feedback and support in this project and related work. We would like to thank Rachael Boyd and Richard Hepworth for interesting conversations at the 2024 British Topology Meeting in Aberdeen. We would like to thank Guy Boyde for his comments and feedback on a previous draft of this paper. We thank the anonymous referee for their helpful and constructive comments.

\subsection*{Conventions}
Throughout, unless otherwise stated, $k$ will be a unital, commutative ring and $n$ will be a positive integer. We will write $\ul{n}$ for the set $\llb 1 ,\dotsc , n\rrb$.

\section{An aggregation of algebras}
\label{alg-sec}
In this section we collect the definitions of the algebras that we will consider in this paper, namely the partition algebras, the Tanabe algebras, the uniform block permutation algebras and the totally propagating partition algebras.

\subsection{Partition algebras}

\begin{defn}
A \emph{partition $n$-diagram} is an undirected graph on two columns of $n$ vertices where each edge is incident to two distinct vertices and there is at most one edge between any two vertices. By convention, the vertices down the left-hand column will be labelled by $1,\dots , n$ in ascending order from top to bottom and the vertices down the right-hand column will be labelled by $\ol{1},\dotsc , \ol{n}$ in ascending order from top to bottom.
\end{defn}

These diagrams are called partition $n$-diagrams because the connected components of the graph determine and are determined by a partition of the set $\llb 1, \ol{1}, \dotsc , n , \ol{n}\rrb$. We say two partition $n$-diagrams are \emph{equivalent} if they determine and are determined by the same partition of the set $\llb 1, \ol{1}, \dotsc , n , \ol{n}\rrb$. Henceforth, when referring to a diagram, we will mean its equivalence class.

\begin{defn}
\label{terminology-defn}
We collect some important terminology for graphs that will be used throughout the rest of the paper.
\begin{enumerate}
    \item An edge that connects the left-hand column of vertices to the right-hand column of vertices will be called a \emph{propagating edge}. 
    \item A connected component which contains vertices in both columns (that is, a connected component containing a propagating edge) will be called a \emph{propagating component}.
    \item An edge that connects two vertices in the same column will be called a \emph{non-propagating edge}. 
    \item A vertex not connected to any other by an edge will be called an \emph{isolated vertex}.
    \item We will refer to the number of vertices in a connected component of a partition $n$-diagram as the \emph{cardinality} of the connected component.
    \item Any diagram having precisely $n$ propagating components will be called a \emph{permutation diagram}. All other diagrams will be referred to as \emph{non-permutation diagrams}.
\end{enumerate}
\end{defn}

\begin{defn}
Let $\delta \in k$. The \emph{partition algebra}, $P_n(\delta)$, is the $k$-algebra with basis consisting of all partition $n$-diagrams with the multiplication defined by the $k$-linear extension of the following product of diagrams. Let $d_1$ and $d_2$ be partition $n$-diagrams. The product $d_1d_2$ is obtained by the following procedure:
\begin{itemize}
    \item Place the diagram $d_2$ to the right of the diagram $d_1$ and identify the vertices $\ol{1},\dotsc , \ol{n}$ in $d_1$ with the vertices $1,\dotsc , n$ in $d_2$. Call this diagram with three columns of vertices $d_1\ast d_2$. We drop the labels of the vertices in the middle column and we preserve the labels of the left-hand column and right-hand column.
    \item Count the number of connected components that lie entirely within the middle column of the new diagram $d_1\ast d_2$. Call this number $\alpha$.
    \item Make a new partition $n$-diagram, $d_3$, as follows. Given distinct vertices $x$ and $y$ in the set $\llb 1, \ol{1}, \dotsc , n , \ol{n}\rrb$, $d_3$ has an edge between $x$ and $y$ if there is a path from $x$ to $y$ in $d_1\ast d_2$.
    \item We define the composite $d_1d_2=\delta^{\alpha}d_3$.
\end{itemize}
We note that this product is associative and well-defined up to equivalence of partition $n$-diagrams (see \cite[Proposition 1]{Martin-partition} for instance). The identity element consists of the diagram with $n$ horizontal edges.
\end{defn}

\begin{eg}
Here is an example of the composition defined above.
Suppose we have diagrams
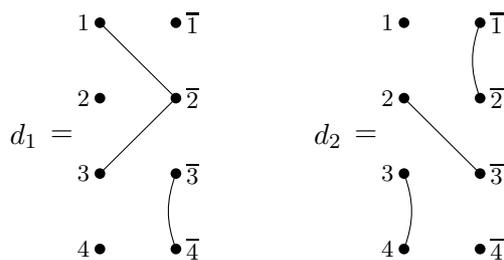
\begin{center}
\begin{tikzpicture}
 \tikzset{>=stealth}
\draw (0,1.5) node {$d_1$};
\draw (0.5,1.5) node {$=$};
\fill (1,0) circle[radius=2pt];
\fill (1,1) circle[radius=2pt];
\fill (1,2) circle[radius=2pt];
\fill (1,3) circle[radius=2pt];
\fill (2,0) circle[radius=2pt];
\fill (2,1) circle[radius=2pt];
\fill (2,2) circle[radius=2pt];
\fill (2,3) circle[radius=2pt];
\draw (1,0) node[left] {\footnotesize $4$};
\draw (1,1) node[left] {\footnotesize $3$};
\draw (1,2) node[left] {\footnotesize $2$};
\draw (1,3) node[left] {\footnotesize $1$};
\draw (2,0) node[right] {\footnotesize $\ol{4}$};
\draw (2,1) node[right] {\footnotesize $\ol{3}$};
\draw (2,2) node[right] {\footnotesize $\ol{2}$};
\draw (2,3) node[right] {\footnotesize $\ol{1}$};
\draw (1,3) -- (2,2) -- (1,1);
\path[-] (2,0) edge [bend left=20] (2,1);   

\draw (4,1.5) node {$d_2$};
\draw (4.5,1.5) node {$=$};
\fill (5,0) circle[radius=2pt];
\fill (5,1) circle[radius=2pt];
\fill (5,2) circle[radius=2pt];
\fill (5,3) circle[radius=2pt];
\fill (6,0) circle[radius=2pt];
\fill (6,1) circle[radius=2pt];
\fill (6,2) circle[radius=2pt];
\fill (6,3) circle[radius=2pt];
\draw (5,0) node[left] {\footnotesize $4$};
\draw (5,1) node[left] {\footnotesize $3$};
\draw (5,2) node[left] {\footnotesize $2$};
\draw (5,3) node[left] {\footnotesize $1$};
\draw (6,0) node[right] {\footnotesize $\ol{4}$};
\draw (6,1) node[right] {\footnotesize $\ol{3}$};
\draw (6,2) node[right] {\footnotesize $\ol{2}$};
\draw (6,3) node[right] {\footnotesize $\ol{1}$};
\draw (5,2) -- (6,1);
\path[-] (6,3) edge [bend right=20] (6,2);
\path[-] (5,1) edge [bend left=20] (5,0);
\end{tikzpicture}
\end{center}
in $P_4(\delta)$. In this case we have 
\begin{center}
\begin{tikzpicture}
\draw (0,1.5) node {$d_1\ast d_2$};
\draw (0.9,1.5) node {$=$};
\fill (1.5,0) circle[radius=2pt];
\fill (1.5,1) circle[radius=2pt];
\fill (1.5,2) circle[radius=2pt];
\fill (1.5,3) circle[radius=2pt];
\fill (2.5,0) circle[radius=2pt];
\fill (2.5,1) circle[radius=2pt];
\fill (2.5,2) circle[radius=2pt];
\fill (2.5,3) circle[radius=2pt];
\fill (3.5,0) circle[radius=2pt];
\fill (3.5,1) circle[radius=2pt];
\fill (3.5,2) circle[radius=2pt];
\fill (3.5,3) circle[radius=2pt];
\draw (1.5,0) node[left] {\footnotesize $4$};
\draw (1.5,1) node[left] {\footnotesize $3$};
\draw (1.5,2) node[left] {\footnotesize $2$};
\draw (1.5,3) node[left] {\footnotesize $1$};
\draw (3.5,0) node[right] {\footnotesize $\ol{4}$};
\draw (3.5,1) node[right] {\footnotesize $\ol{3}$};
\draw (3.5,2) node[right] {\footnotesize $\ol{2}$};
\draw (3.5,3) node[right] {\footnotesize $\ol{1}$};
\draw (1.5,3) -- (2.5,2) -- (3.5,1);
\draw (2.5,2) -- (1.5,1);
\path[-] (2.5,1) edge [bend right=20] (2.5,0);
\path[-] (2.5,1) edge [bend left=20] (2.5,0);
\path[-] (3.5,3) edge [bend right=20] (3.5,2);
\draw (4.5,1.5) node {\text{and}};
\draw (5.5,1.5) node {$d_1d_2$};
\draw (6.2,1.5) node {$=$};
\draw (6.6,1.5) node {$\delta^2$};
\fill (7.5,0) circle[radius=2pt];
\fill (7.5,1) circle[radius=2pt];
\fill (7.5,2) circle[radius=2pt];
\fill (7.5,3) circle[radius=2pt];
\fill (8.5,0) circle[radius=2pt];
\fill (8.5,1) circle[radius=2pt];
\fill (8.5,2) circle[radius=2pt];
\fill (8.5,3) circle[radius=2pt];
\draw (7.5,0) node[left] {\footnotesize $4$};
\draw (7.5,1) node[left] {\footnotesize $3$};
\draw (7.5,2) node[left] {\footnotesize $2$};
\draw (7.5,3) node[left] {\footnotesize $1$};
\draw (8.5,0) node[right] {\footnotesize $\ol{4}$};
\draw (8.5,1) node[right] {\footnotesize $\ol{3}$};
\draw (8.5,2) node[right] {\footnotesize $\ol{2}$};
\draw (8.5,3) node[right] {\footnotesize $\ol{1}$};
\path[-] (8.5,3) edge [bend right=20] (8.5,2);
\path[-] (7.5,3) edge [bend left=20] (7.5,1);
\draw (7.5,3) -- (8.5,1) -- (7.5,1);
\end{tikzpicture}
\end{center}
\end{eg}\textbf{}

We refer the reader to \cite{Jones-partition,BHP2} for some more examples of composing partition diagrams. (We note that some authors work with two rows of vertices rather than two columns and compose diagrams vertically rather than horizontally.)

\subsection{Subalgebras}

We now recall the subalgebras which will be our main focus.

\begin{defn}
 Fix a positive integer $r$. The \emph{Tanabe algebra}, $\mathcal{T}_n(\delta,r)$, is defined as the subalgebra of the partition algebra $P_n(\delta)$ spanned $k$-linearly by those partition $n$-diagrams such that, for each connected component, the difference between the number of vertices in the left and right columns is congruent to zero modulo $r$.    
\end{defn}

\begin{rem}
If we take $r=1$, we recover the partition algebra, $P_n(\delta)$. The algebras $\mathcal{T}_n(\delta,2)$ are sometimes called the \emph{even partition algebras} or the \emph{parity matching algebras} (see \cite{Scrimshaw} for instance). 
\end{rem}

\begin{defn}
The \emph{uniform block permutation algebra}, $U_n$, is the subalgebra of $P_n(\delta)$ spanned $k$-linearly by the partition $n$-diagrams such that each connected component has the same number of vertices in the left-hand column as it does in the right-hand column.   
\end{defn}

\begin{defn}
The \emph{totally propagating partition algebra}, $TPP_n$, is the subalgebra of $P_n(\delta)$ spanned $k$-linearly by the partition $n$-diagrams such that every connected component is propagating.        
\end{defn}

\begin{rem}
For uniform block permutation algebras and totally propagating partition algebras, we note that in the procedure for composing two $n$-diagrams $d_1$ and $d_2$, the diagram $d_1\ast d_2$ can have no connected components that lie entirely within the middle column so we drop the parameter $\delta$ from the notation.  
\end{rem}

\subsection{Augmentations}

Recall that a $k$-algebra is said to be \emph{augmented} if it comes equipped with a $k$-algebra map $\varepsilon\colon A \rightarrow k$, which is called the \emph{augmentation}.

Recall from \cite[Section 1]{BHP2} that the partition algebras $P_n(\delta)$ can be equipped with an augmentation that sends the permutation diagrams to $1 \in k$ and all non-permutation diagrams to $0\in k$ and that we define the trivial module $\1$ to be a copy of $k$ where $P_n(\delta)$ acts via the augmentation.

We see immediately that the Tanabe algebras, uniform block permutation algebras and totally propagating partition algebras are augmented by restricting the augmentation of $P_n(\delta)$ along the subalgebra inclusion maps. We can therefore define trivial modules for these three families of algebras similarly.

\section{Cohomology of algebras: proving Theorem \ref{theorem-B}}
\label{Theorem-B-section}

We recall the definitions of $k$-free idempotent left cover and the Mayer--Vietoris complex from \cite{Boyde2} and prove the cohomological analogue of Theorem 1.7 in loc.~cit.

\subsection{Idempotent left covers and the Mayer--Vietoris complex}

The material in this subsection comes from \cite[Sections 1 and 2]{Boyde2}. 

\begin{defn}
Let $A$ be a $k$-algebra. Let $I$ be a two-sided ideal of $A$. Let $w\geqslant h \geqslant 1$. An \emph{idempotent left cover of $I$ of height $h$ and width $w$} is a collection of left ideals $J_1,\dotsc , J_w$ in $A$ such that
\begin{itemize}
    \item $J_1+\cdots +J_w=I$;
    \item for $S\subset \ul{w}$ with $\llv S\rrv \leqslant h$, the intersection
    \[\bigcap_{i\in S} J_i\]
    is either zero or is a principal left ideal generated by an idempotent.
\end{itemize}
If $I$ is free as a $k$-module, then an idempotent left cover is said to be \emph{$k$-free} if there is a choice of $k$-basis for $I$ such that each $J_i$ is free on a subset of this basis.
\end{defn}

\begin{defn}
Let $A$ be a $k$-algebra. Let $I\subset A$ be a two-sided ideal. Let $J_1,\dotsc , J_w$ be an idempotent left cover of $I$. The \emph{Mayer--Vietoris complex associated to the idempotent left cover}, $C_{\star}$, is the chain complex of left $A$-modules defined as follows. We set
\[C_p= \underset{\llv S \rrv = p}{\bigoplus_{S \subset \ul{w}}}\bigcap_{i\in S} J_i\]
for $1\leqslant p \leqslant w$. We set $C_0=A$, $C_{-1}=A/I$ and so $C_n=0$ for $n>w$ and $n<-1$.

The differential $C_0\rightarrow C_{-1}$ is the projection map $A\rightarrow A/I$. The differential $C_1\rightarrow C_0$ is the direct sum of the inclusion of the left ideals $J_i\rightarrow A$. For $p\geqslant 2$, the differential $C_p\rightarrow C_{p-1}$ is defined on the summand $\cap_{i\in S} J_i$ by
\[x\mapsto \sum_{j\in S} (-1)^{\#(S,j)} i_{(S,j)}(x)  \]
where $\#(S,j)$ is the number of elements of $S$ that are less than $j$ and $i_{(S,j)}$ is the inclusion
\[\bigcap_{i\in S} J_i \rightarrow \bigcap_{i\in S\setminus \llb j\rrb} J_i.\]
\end{defn}

Recall that for a left $A$-module $M$, a \emph{partial projective resolution of length $h$} of $M$ by left $A$-modules is an exact sequence
\[P_h\rightarrow P_{h-1}\rightarrow \cdots \rightarrow P_0\rightarrow M \rightarrow 0\]
where each $P_i$ is a projective left $A$-module.

The Mayer--Vietoris complex satisfies the following important property \cite[Proposition 2.4]{Boyde2}.

\begin{prop}
\label{Boyde-prop}
Let $A$ be a $k$-algebra. Let $I$ be a two-sided ideal of $A$. Let $J_1,\dotsc , J_w$ be a $k$-free idempotent left cover of $I$ of height $h$.

The truncation, $C_{\star}^{\leqslant h}$ of the Mayer--Vietoris complex associated to the idempotent left cover is a length $h$ partial projective resolution of $A/I$ by left $A$-modules with $C_0=A$. The partial projective resolution has the additional property that $X\otimes_A C_p^{\leqslant h}=0$ for $p\geqslant 1$ for any right $A$-module $X$ on which $I$ acts as multiplication by $0\in k$. If $h=w$ then $C_{\star}^{\leqslant h}=C_{\star}$ is a projective resolution of $A/I$ by left $A$-modules. 
\fakeqed
\end{prop}

\subsection{A stable isomorphism on cohomology}

We now prove the proposition from which we can deduce Theorem \ref{theorem-B}.

\begin{prop}
\label{partition-technical-prop}
Let $M$ and $N$ be right $A$-modules. Let $I$ be a two-sided ideal that acts as multiplication by $0\in k$ on $M$ and $N$.

Suppose that there exists a partial projective resolution of length $h$, $C_{\star}^{\leqslant h}$, of $A/I$ by left $A$-modules such that $C_0^{\leqslant h}=A$ and such that $X\otimes_A C_p^{\leqslant h}=0$ for $p\geqslant 1$ for any right $A$-module $X$ on which $I$ acts as multiplication by $0\in k$.

There is a natural isomorphism of $k$-modules
\[\Ext_A^q(M,N) \cong \Ext_{A/I}^q(M,N)\]
for $q\leqslant h$. Furthermore, the natural map
\[\Ext_{A/I}^{h+1}(M,N)\rightarrow \Ext_{A}^{h+1}(M,N)\]
is an injection.
\end{prop}
\begin{proof}
We follow a similar argument to \cite[Theorem 2.7]{Boyde2} throughout. When $q\leqslant h$, we will show that $\Ext_A^q(M,N)$ and $\Ext_{A/I}^q(M,N)$ are the cohomology of the same cochain complex.

Let $F_{\star}$ be a free resolution of $M$ by right $A$-modules. We know that $\Ext_A^{\star}(M,N)$ is the cohomology of the cochain complex $\Hom_A(F_{\star},N)$.

Since $I$ acts as multiplication by $0\in k$ on $N$, we have an isomorphism of cochain complexes
\[\Hom_A(F_{\star},N)\cong \Hom_{A/I}\left(F_{\star}\otimes_A (A/I) , N\right)\]
by extension and restriction of scalars.

We observe that since each $F_{i}$ is free as an $A$-module, each $F_{i}\otimes_A (A/I)$ is free as an $A/I$-module.

In order to deduce the isomorphisms for $q\leqslant h$, it suffices to show that the homology of $F_{\star}\otimes_A (A/I)$ is isomorphic to $M$ in degree zero and $0$ in degrees $0< q\leqslant h$. 

By assumption, $C_{\star}^{\leqslant h}$ is a partial projective resolution of $A/I$ by left $A$-modules with the property that $X\otimes_A C_p^{\leqslant h}=0$ for $p\geqslant 1$ for any right $A$-module $X$ on which $I$ acts as multiplication by $0\in k$. Furthermore, $C_0^{\leqslant h}=A$. 

Therefore, by \cite[Proposition 2.6]{Boyde2} for example, we have 
\[H_{q}(F_{\star}\otimes_A (A/I))=\Tor_q^A(M,A/I)=0\]
for $0<q\leqslant h$ and
\[H_0(F_{\star}\otimes_A (A/I))=\Tor_0^A(M,A/I)=M\otimes_A (A/I)=M\otimes_A C_0^{\leqslant h}= M\otimes_A A \cong M.\]

This yields the necessary isomorphisms for $q\leqslant h$.

We now turn our attention to the injection $\Ext_{A/I}^{h+1}(M,N)\rightarrow \Ext_{A}^{h+1}(M,N)$. 

Consider our chain complex $F_{\star}\otimes_A (A/I)$. We can take the direct sum of $F_{h+2}\otimes_A (A/I)$ with a free right $A/I$-module $X$ such that
\[(F_{h+2}\otimes_A (A/I))\oplus X \rightarrow F_{h+1}\otimes_A (A/I)\rightarrow F_h\otimes_A (A/I)\]
is exact (by killing off the kernel of the second map). Call this new complex $(F_{\star}\otimes_A (A/I))\oplus X_{(h+2)}$.

By construction, the homology of this new complex is isomorphic to $M$ in degree zero and is $0$ in degrees $0<q\leqslant h+1$.

Furthermore, we have a short exact sequence of chain complexes of right $(A/I)$-modules
\[0\rightarrow F_{\star}\otimes_A (A/I) \rightarrow (F_{\star}\otimes_A (A/I))\oplus X_{(h+2)} \rightarrow X_{(h+2)} \rightarrow 0,\]
where $X_{(h+2)}$ is the complex given by $X$ concentrated in degree $h+2$.

Applying $\Hom_{A/I}(-,N)$ to this short exact sequence and taking the long exact sequence in cohomology we recover the isomorphisms in degrees $q\leqslant h$. Furthermore,
\begin{itemize}
    \item the cohomology of $X_{(h+2)}$ is zero in degree $h+1$ and is $X$ in degree $h+2$,
    \item the cohomology group in degree $h+1$ of \[\Hom_{A/I}((F_{\star}\otimes_A (A/I))\oplus X_{(h+2)},N)\] is $\Ext_{A/I}^{h+1}(M,N)$ and
    \item the cohomology group in degree $h+1$ of $\Hom_{A/I}(F_{\star}\otimes_A (A/I),N)$ is $\Ext_{A}^{h+1}(M,N)$ by the isomorphism of chain complexes given by the extension and restriction of scalars.
\end{itemize}
Thus the long exact sequence in cohomology yields an exact sequence
\[0\rightarrow \Ext_{A/I}^{h+1}(M,N) \rightarrow \Ext_{A}^{h+1}(M,N) \rightarrow X,\]
from which it follows that $\Ext_{A/I}^{h+1}(M,N) \rightarrow \Ext_{A}^{h+1}(M,N)$ is an injection.
\end{proof}

\begin{proof}[Proof of Theorem \ref{theorem-B}]
Theorem \ref{theorem-B} follows from Proposition \ref{partition-technical-prop} and Proposition \ref{Boyde-prop} by taking $M=N=\1$.   
\end{proof}

\section{Cohomology of partition algebras and Jones annular algebras}
\label{Cohom-partition-sec}

In this section we show that for any $\delta \in k$, the cohomology of the partition algebras is stably isomorphic to the group cohomology of the symmetric groups using Theorem \ref{theorem-B}. As a corollary, we deduce that the partition algebras exhibit cohomological stability. Furthermore we show that if $\delta$ is invertible, then the cohomology of the partition algebras is globally isomorphic to the group cohomology of the symmetric groups. We also use Theorem \ref{theorem-B} to deduce the cohomological version of Boyde's result on the Jones annular algebras.

\subsection{Partition algebras}
We begin by showing that the cohomology of the partition algebras is stably isomorphic to the cohomology of the symmetric groups.

Let $I_{n-1}$ be the two-sided ideal of $P_n(\delta)$ spanned $k$-linearly by the non-permutation diagrams. One can see that this is a two-sided ideal by checking that the composite of two diagrams with $i$ propagating components and $j$ propagating components respectively is a scalar multiple of a diagram with at most $\min(i,j)$ propagating components.

\begin{defn}
 For $i\in \ul{n}$, let $K_i$ denote the left ideal in $P_n(\delta)$ spanned $k$-linearly by the diagrams where the vertex $\ol{i}$ is an isolated vertex. For distinct $i$ and $j$ in $\ul{n}$ with $i<j$, we let $L_{i,j}$ denote the left ideal in $P_n(\delta)$ spanned $k$-linearly by the diagrams where $\ol{i}$ and $\ol{j}$ are in the same connected component.  
\end{defn}

\begin{thm}
\label{partition-cohom-stability-thm}
There is a natural isomorphism of $k$-modules
\[\Ext_{P_n(\delta)}^q(\mathbbm{1},\mathbbm{1}) \cong H^{q}(\Sigma_n ,\mathbbm{1})\]
for $q\leqslant n-1$.
\end{thm}
\begin{proof}
There is an isomorphism of $k$-algebras $P_n(\delta)/I_{n-1} \cong k[\Sigma_n]$. 

Boyde (\cite[Section 3]{Boyde2}) shows that the left ideals $K_i$ and $L_{i,j}$ form a $k$-free idempotent left cover of $I_{n-1}$ of height $n-1$. The theorem now follows from Theorem \ref{theorem-B}.
\end{proof}

As a corollary, we can deduce the cohomological analogue of \cite[Corollary C]{BHP2}.

\begin{cor}
\label{Partition-stability-cor}
 The inclusion map $P_{n-1}(\delta) \rightarrow P_n(\delta)$ induces a map on cohomology
 \[\Ext_{P_n(\delta)}^q(\1 ,\1) \rightarrow \Ext_{P_{n-1}(\delta)}^q(\1 ,\1)\]
 which is an isomorphism in degrees $n\geqslant 2q+1$, and this stable range is sharp. Furthermore,
 \[\underset{n\rightarrow \infty}{\lim}\, \Ext_{P_n(\delta)}^{\star}(\1 ,\1) \cong \underset{n\rightarrow \infty}{\lim}\, H^{\star}\left(\Sigma_n , \1\right).\]
\end{cor}
\begin{proof}
This follows from Theorem \ref{partition-cohom-stability-thm} and Nakaoka's result on the cohomological stability of the symmetric groups \cite{Nakaoka}.    
\end{proof}

We now turn our attention to the case where $\delta$ is invertible.
\begin{lem}
\label{delta-invertible-lem}
When $\delta$ is invertible in $k$, Boyde's $k$-free idempotent cover of $I_{n-1}\subset P_n(\delta)$ has height equal to its width.
\end{lem}
\begin{proof}
As noted, Boyde shows that the left ideals $K_i$ and $L_{i,j}$ form a $k$-free idempotent left cover of $I_{n-1}$ of height $n-1$. In fact, the only non-zero intersection of these ideals which is not, in general, principal and generated by an idempotent is the intersection
\[J=\bigcap_{i=1}^n K_i.\]
This ideal is spanned $k$-linearly by all partition $n$-diagrams such that every vertex in the right-hand column is isolated.

Let $e$ be the partition $n$-diagram such that all $2n$ vertices are isolated. Since $\delta$ is invertible we may consider $\delta^{-n}e\in P_n(\delta)$. It is immediate that this is idempotent since $e^2=\delta^n e$. Furthermore, right multiplication by $\delta^{-n}e$ yields a map $P_n(\delta)\rightarrow J$ since, on basis diagrams, right multiplication by $e$ yields diagrams whose right-hand column consists of isolated vertices. This map is the identity when restricted to $J$. Let $d\in J$ be a basis diagram, so the right-hand column consists only of isolated vertices. Then  
\[d\cdot (\delta^{-n}e)=\delta^{-n}(d\cdot e)= \delta^{-n}(\delta^n d)=d.\]
In other words, when $\delta$ is invertible, $J$ is principal and generated by idempotent.
\end{proof}

\begin{cor}
Suppose $\delta \in k$ is invertible. There exist isomorphisms of graded $k$-modules
\[\Tor_{\star}^{P_n(\delta)}\left(\1 , \1\right) \cong H_{\star}\left(\Sigma_n , \1\right)\quad \text{and} \quad \Ext_{P_n(\delta)}^{\star}\left(\1 , \1\right) \cong H^{\star}\left(\Sigma_n , \1\right).\]
\end{cor}
\begin{proof}
Lemma \ref{delta-invertible-lem} tells us that when $\delta$ is invertible, the ideals $K_i$ and $L_{i,j}$ form a $k$-free idempotent cover of $I_{n-1}$ with height equal to the width. The isomorphisms of graded $k$-modules now follow from Theorem \ref{theorem-B}.
\end{proof}

\begin{rem}
    We note that the homological statement recovers \cite[Theorem A]{BHP2}.
\end{rem}

\subsection{Jones annular algebras}
Boyde uses his theorem to prove results about the homology of \emph{Jones annular algebras}, $J_n(\delta)$. The Jones annular algebras can be defined as the $k$-linear span of partition $n$-diagrams such that each connected component has size two and the diagram can be represented as a planar graph on an annulus (see \cite{Jones-annular-1, GL-cellular, Boyde2} for precise definitions). 

The Jones annular algebras are augmented by sending the diagrams that represent cyclic permutations to $1\in k$ and all other diagrams to $0\in k$. In particular, recalling the two-sided ideal $I_{n-1}\subset P_n(\delta)$, we have an isomorphism of $k$-algebras $J_n(\delta)/(J_n(\delta)\cap I_{n-1})\cong k[C_n]$, where $C_n$ is the cyclic group of order $n$.

\begin{thm}
\label{Jones-ann-alg-thm}
Let $\delta \in k$. Let $C_n$ denote the cyclic group of order $n$. There is a natural isomorphism 
\[\Ext_{J_n(\delta)}^{q}\left(\1 , \1\right) \cong  H^{q}(C_n,\1)\]
for $q\leqslant \frac{n}{2}-1$. 

Furthermore, if $n$ is odd or if $\delta$ is invertible, the isomorphism holds for all $q$.
\end{thm}
\begin{proof}
For $1\leqslant i \leqslant n-1$, let $\mathcal{J}_i$ denote the left ideal of $J_n(\delta)$ spanned $k$-linearly by basis diagrams such that the vertices $\ol{i}$ and $\ol{i+1}$ are connected. Let $\mathcal{J}_n$ denote the left ideal of $J_n(\delta)$ spanned $k$-linearly by basis diagrams such that the vertices $\ol{n}$ and $\ol{1}$ are connected. Boyde \cite[Section 4]{Boyde2} shows that the ideals $\mathcal{J}_i$ for $1\leqslant i \leqslant n$ form a $k$-free idempotent left cover of the two-sided ideal $J_n(\delta)\cap I_{n-1}\subset J_n(\delta)$ with height $\frac{n}{2}-1$. Boyde also shows that when $\delta$ is invertible or when $n$ is odd, the cover has height equal to the width. Combining Boyde's work with Theorem \ref{theorem-B} yields the result.
\end{proof}

\section{(Co)homology of Tanabe algebras, totally propagating partition algebras and the uniform block algebras}
\label{Tanabe-sec}

Recall that $I_{n-1}\subset P_n(\delta)$ is the two-sided ideal spanned $k$-linearly by the non-permutation diagrams.

Furthermore, recall that $L_{i,j}\subset P_n(\delta)$ is the left ideal spanned $k$-linearly by partition $n$-diagrams such that $\ol{i}$ and $\ol{j}$ are in the same connected component.

\begin{defn}
 Let $\nu_{a,b}\in P_n(\delta)$ be the partition $n$-diagram whose connected components are $\lbrace a , b ,\ol{a} , \ol{b}\rbrace$ and $\lbrace i, \ol{i}\rbrace$ for $i\in \ul{n}\setminus \llb a , b\rrb$.  
\end{defn}

\begin{lem}
\label{idempotent-cover-A-lem}
 Let $A$ be a subalgebra of $P_n(\delta)$ such that
 \begin{itemize}
    \item $A$ is spanned $k$-linearly by a basis of diagrams,
     \item $A$ contains no diagrams with isolated vertices, and
     \item $A$ contains all the elements of the form $\nu_{a,b}$.
 \end{itemize}
 Then the left ideals $A\cap L_{i,j}$ form a $k$-free idempotent cover of the two-sided ideal $A\cap I_{n-1}$ whose height is equal to its width.
\end{lem}
\begin{proof}
We begin by showing that the left ideals $A\cap L_{i,j}$ cover the two-sided ideal $A\cap I_{n-1}$. 

If a basis diagram lies in $A\cap L_{i,j}$, it contains a connected component with at least two vertices in the right-hand column and so can have at most $n-1$ propagating components. Therefore, the basis diagram lies in $A\cap I_{n-1}$. Conversely, a basis diagram in $A\cap I_{n-1}$ contains at most $n-1$ propagating components. Since we cannot have isolated vertices, this means that at least two vertices in the right-hand column must be in the same connected component. Hence the diagram must lie in some $A\cap L_{i,j}$.

We now show that any intersection of ideals $A\cap L_{i,j}$ must be zero or principal and generated by an idempotent. We break this up into parts.
\begin{enumerate}
\item Let $\ul{n}_{<}^2$ be the set of indices $(i,j)$ with $1\leqslant i< j \leqslant n$. Let $T\subset \ul{n}_{<}^2$. Let
\[J= \bigcap_{(i,j)\in T} A\cap L_{i,j}.\]

We claim that $J\cdot \nu_{a,b} \subset (A\cap L_{a,b}) \cap J$. We begin by noting that all diagrams of the form $\nu_{a,b}$ are in $A$ by assumption.
\item Let $\rho$ be a basis diagram in $J$. We must show that $\rho \nu_{a,b} \in (A\cap L_{a,b}) \cap J$. Since $\rho \in J$, for each $(i,j)\in T$, the vertices $\ol{i}$ and $\ol{j}$ are connected in $\rho$. The vertices $i$ and $\ol{i}$ are connected in $\nu_{a,b}$, as are $j$ and $\ol{j}$. Therefore, we see that $\ol{i}$ and $\ol{j}$ are connected in the composite $\rho\nu_{a,b}$ so $\rho\nu_{a,b} \in J$. Since $A\cap L_{a,b}$ is a left ideal and $\nu_{a,b}\in A\cap L_{a,b}$ by assumption, the composite $\rho\nu_{a,b} \in A \cap L_{a,b}$.
\item We observe that right multiplication by $\nu_{a,b}$ gives a retraction of the inclusion map $(A\cap L_{a,b}) \cap J \rightarrow J$. Right multiplication by $\nu_{a,b}$ acts on an $n$-diagram, $d$, by merging the component of $d$ containing $\ol{a}$ with the component containing $\ol{b}$, whilst preserving all other connected components of $d$. For any $\rho \in (A \cap L_{a,b})$, $\ol{a}$ and $\ol{b}$ already lie in the same connected component and so $\rho \nu_{a,b} = \rho$.
\item By repeatedly applying the argument of the previous two points, we see that the composite
\[J\hookrightarrow A \xrightarrow{\Pi} J,\]
where $\Pi$ is right multiplication by the product of all $\nu_{i,j}$ for $(i,j)\in T$, is the identity map. Since $J$ is a left $A$-module retract of $A$ itself, it then follows that $J$ 
is principal and generated by an idempotent
\end{enumerate}
Therefore, the family of left ideals $A \cap L_{i,j}$ forms an idempotent left cover of $A \cap I_{n-1}$ whose height is equal to its width.
\end{proof}

We now prove Theorem \ref{Tanabe-thm}.

\begin{proof}[Proof of Theorem \ref{Tanabe-thm}]
 We observe that we have isomorphisms of $k$-algebras 
 \begin{itemize}
     \item $\mathcal{T}_n(\delta,r)/(\mathcal{T}_n(\delta,r)\cap I_{n-1}) \cong k[\Sigma_n]$;
     \item $U_n/(U_n\cap I_{n-1}) \cong k[\Sigma_n]$
     and 
     \item $TPP_n/(TPP_n\cap I_{n-1}) \cong k[\Sigma_n]$.
\end{itemize}
By definition, the basis diagrams in $\mathcal{T}_n(\delta,r)$, $TPP_n$ and $U_n$ cannot contain any isolated vertices. We also note that all three algebras contain all the elements of the form $\nu_{a,b}$, since all connected components of each $\nu_{a,b}$ are propagating, with an equal number of vertices in each column. Theorem \ref{Tanabe-thm} now follows from Lemma \ref{idempotent-cover-A-lem} and Theorem \ref{theorem-B}.
 \end{proof}

\bibliographystyle{alpha}
\bibliography{tanabe-refs}

\end{document}